\documentclass{adamjoucc}
\usepackage{lastpage}
\usepackage{color}
\usepackage{amssymb}
\usepackage{mathtools}
\usepackage{tikz}
\usepackage{chemfig}
\usepackage{ogonek} % for polish names

\journalyear{2023}
\journalpages{1}{\pageref{LastPage}}
\evenrunninghead{C.\ Flamm, S.\ M{\"u}ller, P.F.\ Stadler}
\oddrunninghead{C.\ Flamm, S.\ M{\"u}ller, P.F.\ Stadler:
  Atom-Atom Maps and Electron Pushing Diagrams}
\oddrunninghead{Atom-Atom Maps and Electron Pushing Diagrams}
\makeatletter
\def\@journalname{Preprint}%{Discrete Mathematical Chemistry}
\def\@issn{ }
\def\@journalnumber{}%{(for review)}

\makeatother

\DeclareMathOperator{\ch}{ch}
\DeclareMathOperator{\val}{val}
\DeclareMathOperator{\sign}{sgn}

\newcommand{\R}{\mathbb R}

\definecolor{orange}{rgb}{1,0.5,0}
\newcommand{\ue}{ü} 

\begin{document}

\begin{frontmatter}

\titledata{Every atom-atom map can be explained by electron pushing
  diagrams}

\authordata{Christoph Flamm}{Department of Theoretical Chemistry,
  University of Vienna, W{\"a}hringerstra{\ss}e 17, A-1090 Wien,
  Austria}{xtof@tbi.univie.ac.at}{}

\authordata{Stefan M{\ue}ller}{Faculty of Mathematics,
  University of Vienna, Oskar-Morgenstern-Platz 1, A-1090 Wien,
  Austria}{st.mueller@univie.ac.at}{}
        
\authordata{Peter F.\ Stadler}{Bioinformatics Group,
  Department of Computer Science \&
  Interdisciplinary Center for Bioinformatics \&
  Center for Scalable Data Analytics and Artificial Intelligence 
  Dresden/Leipzig \& 
  School of Embedded Composite Artificial Intelligence,
  Leipzig University, H{\"a}rtelstra{\ss}e 16–18,
  D-04107 Leipzig, Germany;
  Max Planck Institute for Mathematics in the Sciences,
  Inselstra{\ss}e 22, D-04103 Leipzig, Germany;
  Department of Theoretical Chemistry,
  University of Vienna, W{\"a}hringerstra{\ss}e 17,
  A-1090 Wien, Austria;
  Facultad de Ciencias, Universidad National de Colombia;
  Bogot{\'a}, Colombia;
  Santa Fe Institute, 1399 Hyde Park Rd., Santa Fe 
  NM 87501, USA}{studla@bioinf.uni-leipzig.de}{}

%  Complex Reactions;

\msc{ }

\begin{abstract}
  Chemical reactions can be understood as transformations of multigraphs
  (molecules) that preserve vertex labels (atoms) and degrees (sums of
  bonding and non-bonding electrons), thereby implying the atom-atom map of
  a reaction.  The corresponding reaction mechanism is often described by
  an electron pushing diagram that explains the transformation by
  consecutive local relocations of invidudal edges (electron pairs). Here,
  we show that every degree-preserving map between multigraphs, and thus
  every atom-atom map, can be generated by cyclic electron
  pushing. Moreover, it is always possible to decompose such an explanation
  into electron pushing diagrams involving only four electron pairs. This
  in turn implies that every reaction can be decomposed into a sequence of
  elementary reactions that involve at most two educt molecules and two
  product molecules. Hence, the requirement of a mechanistic explantion in
  terms of electron pushing and small imaginary transition states does not
  impose a combinatorial constraint on the feasibility of hypothetical
  chemical reactions.
\end{abstract}

\end{frontmatter}

\section{Introduction}

Chemical reaction networks can be viewed abstractly as directed hypergraphs
with compounds as vertices and reactions as directed hyperedges. However,
not every directed hypergraph has a chemical interpretation. In particular,
the conservation of atoms and thus mass implies non-trivial constraints.
Most importantly, they ensure the existence of atom-atom maps that
guarantee that reactions can be written as re-assignments of bonds
\cite{Mueller:22a}. Here, we ask whether there are constraints on
chemically feasibly atom-atom maps that derive from the concept of electron
pushing diagrams, a low-level mechanistic description of chemical reactions
as stepwise local relocations of electron pairs. We shall see that this is
not the case: at least mathematically, every atom-atom map admits such a
low-level mechanistic explanation.

A structural formula represents a chemical species as a (connected) graph,
whose vertices are labeled by atom types and edges refer to chemical
bonds. \emph{Lewis formulas} \cite{Lewis:1916} are equivalent to
vertex-labeled multigraphs in which each bonding electron pair is
represented as an individual edge, and each non-bonding electron pair as a
loop. The electron pair of a bond is considered to be divided up between
the two atoms that it connects, while a non-bonding pair is localized at a
single atom. This representation of molecules agrees with the matrix
formalism of Dugundji \& Ugi \cite{Dugundji:73}.  Since the number of
electrons in the outer shell is usually preserved, the atom type defines
the degree of a vertex in the multigraph, matching Frankland's
``atomicity'' and conforming to the IUPAC term ``valency''
\cite{Muller:94}.

When molecules are modeled as Lewis formulas, chemical reaction mechanisms
can be described as \emph{electron pushing diagrams} (EPD)
\cite{Kermack:1922}. Their simplest form describes the movement of
electrons in terms of stepwise local movements of electron pairs. In terms
of conventional chemical notation, this amounts to considering only
``curved arrows'' in the conventional chemical notation
\cite{Alvarez:11}. The elementary step in an EPD is the replacement of an
electron pair (edge) $xy$ by an electron pair (edge) $yz$. Of course, we
require $x\ne z$ since otherwise no change in the molecule would have
incurred. However, we may have either $x=y$ or $y=z$, corresponding to
transitions that convert a non-bonding to a bonding electron pair or
\textit{vice versa}. Since this operation increases the degree at $z$ and
(decreases the degree at $x$), an electron pair at $z$ needs to move in the
next step in order to re-establish valency at $z$. This process progresses
until the degree deficit at $x$ is eventually compensated, see
Fig.~\ref{fig:DA}. The reaction therefore can be understood as a cyclic
sequence of alternating steps of deleting and inserting consecutive edges
in the molecular graph. By construction, the application of an EPD
preserves the number of electrons at each atom, i.e., the vertex degree.

\begin{figure}
  \begin{tabular}{cc}
    \begin{minipage}{0.55\textwidth}
      %% !!!! run ./standaloneFig.sh figDielsAlder to get figDielsAlder.pdf !!!!
      \includegraphics[width=\linewidth]{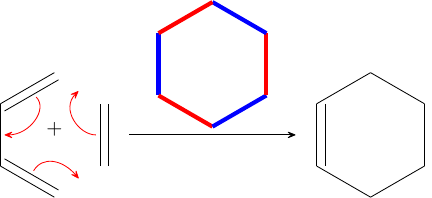}
    \end{minipage}
    &
    \begin{minipage}{0.40\textwidth}
      \caption{Electron pushing representation of a Diels-Alder reaction.
        Above the reaction arrow, the difference or Imaginary Transition
        State (ITS) graph $\Delta$ is shown with blue edges for
        $\sigma(e)=+1$ and red edges for $\sigma(e)=-1$. Tails and heads of
        electron pushing arrows translate to red and blue edges in the ITS,
        respectively.}
      \label{fig:DA}
    \end{minipage}    
  \end{tabular}
\end{figure} 

In a non-mechanistic setting, chemical reactions are modeled in terms of
their atom-atom maps (AAM), i.e., the one-to-one correspondence of atoms in
the educt (reactant) multigraph and the product multigraph. The AAM
implicitly describes the net change of bonds between educts and
products. In the terminology of chemical reaction network theory, the educt
and product graphs $G$ and $H$ correspond to \emph{complexes}
\cite{Horn:72a}, and their connected components are the educt and product
molecules. AAMs thus are fine-grained descriptions of the (directed) edges
in the complex-reaction graph. The net effect of a chemical reaction is
captured by a ``difference multigraph'' in which %new bonds and changes of
bond orders are recorded. In essence, this graph is a graph-theoretic
version of the R-matrix in the formalism of Dugundji and Ugi
\cite{Dugundji:73} and can be seen as variation of the Imaginary Transition
State (ITS) graph \cite{Fujita:86} and Condensed Graph of a Reaction (CGR)
\cite{Hoonakker:11}, from which it differs only by the use of signed
multiple edges instead of specific edge-labels.

The application of an EPD to an educt multigraph $G$ produces a multigraph
$H$ on the same vertex set $V$, and the corresponding AAM is the identity
on $V=V(G)=V(H)$.  Since every push (movement of an electron pair) is
eventually compensated, the EPD model of a reaction preseve the degree
(valency) at each vertex. In this contribution we ask whether the converse
is also true: Is there always a sequence of EPDs that explains a given
degree-preserving AAM? Indeed, we shall show that the difference multigraph
of a degree-preserving AAM can always be explained by a set of disjoint
EPDs.

\section{Mathematical Preliminaries}

Let $G=(V,E)$ be a finite undirected multigraph and denote by
$\mathfrak{C}(G)$ its connected components. We consider the elements of $E$
as distinct; nevertheless, we will write ``an edge $xy$'' to mean an
``element $e\in E$ incident to $x,y\in V$''. For a pair of vertices $x,y\in
V$ we define $m(x,y)=m(y,x)$ as the number of edges connecting $x$ and
$y$. In particular, $m(x,y)=0$ if $x$ and $y$ are not adjacent. Similarly,
$m(x,x)$ denotes the multiplicity of loops at vertex $x$. Thus $m(x,x)=0$
if there is not loop at $x\in V$.  A simple graph is a multigraph without
loops and multiple edges, i.e., $m(x,x)=0$ and $m(x,y)\in\{0,1\}$ for all
$x\ne y$. Following the usual convention for multigraphs, the degree of a
vertex is defined as
\begin{equation}
  \deg(x) \coloneqq  \sum_{y\in V\setminus\{x\}} m(x,y) + 2 m(x,x) .
\end{equation}
Loops count twice because both of their endpoints are incident with $x$.
Chemically, we can interpret $m(x,y)$ with $x\ne y$ as the bond order,
while $m(x,x)$ denotes the number of non-bonding electron pairs at
$x$. Thus, the degree $\deg(x)$ describes the number of electrons at an
atom, which will consider to be a constant.

Lewis formulas have vertex labels designating the atom type and atom-atom
maps (AAM) are defined as label-preserving bijections $\alpha\colon V(G)\to
V(H)$ between two Lewis formulas $G$ and $H$.  For the following
mathematical results, however, the labels themselves are not relevant. We
will only make use of the fact that the total number of electrons in the
outer shell of atom is consider to be fixed and thus reactions preserve the
vertex degree. Instead of atom-atom maps, we therefore consider
degree-preserving bijections $\alpha\colon V(G)\to V(H)$ such that
$\deg_H(\alpha(x))=\deg_G(x)$ for all $x\in V(G)=V(H)$.

Our key construction is a graph-theoretical reformulation of the R-matrix
in the theory of Dugundji \& Ugi \cite{Dugundji:73}:
\begin{definition}
  Let $G$ and $H$ be multigraphs and $\alpha\colon V(G)\to V(H)$ be a
  degree-preserving bijection.  Moreover, for vertices $x$ and $y$, let
  $\delta(x,y)= m_H(\alpha(x),\alpha(y))-m_G(x,y)$, $m(x,y)\coloneqq
  |\delta(x,y)|$, and $\sigma(x,y)\coloneqq \sign(\delta(x,y))$, where
  $\sign\colon \R \to \{-1,0,+1\}$ is the sign function.  Then, the
  \emph{difference multigraph} $\Delta(G,H,\alpha)$ has vertex set
  $V=V(G)=V(H)$ and $m(x,y)\coloneqq |\delta(x,y)|$ edges connecting $x$
  and $y$.  An edge $xy$ is labeled by $\sigma(x,y)$.
\end{definition}
In particular, $\Delta(G,H,\alpha)$ contains no edge between $x$ and $y$ if
the multiplicity of the edges between two vertices $x$ and $y$ in $G$
coincides with the multiplicity of the edges between the images $\alpha(x)$
and $\alpha(y)$ in $H$. The difference multigraph describes the net change
of bond orders that result from the chemical reaction $\mathfrak{C}(G)\to
\mathfrak{C}(H)$. The multiplicities in $G$, $H$, and $\Delta\coloneqq
\Delta(G,H,\alpha)$ satisfy
\begin{equation}
  m_H(\alpha(x),\alpha(y))-m_G(x,y) =
  \sigma_{\Delta}(x,y) \cdot m_{\Delta}(x,y) ,
\end{equation}
for all $x,y$.  We note that setting $B_{xy}\coloneqq m_G(x,y)$,
$E_{xy}\coloneqq m_H(\alpha(x),\alpha(y))$, and $R_{xy}\coloneqq
\sigma_{\Delta}(x,y) \cdot m_{\Delta}(x,y)$ recovers the representation of
a reaction as $B+R=E$ in \cite{Dugundji:73}.

In the following, it will be useful to count the increasing and decreasing
bond orders separately. To this end we introduce
\begin{equation}
  d_+'(x) \coloneqq \sum_{\substack{y\ne x\\\sigma(x,y)=+1}} m_{\Delta}(x,y)
  \quad \text{ and } \quad 
  d_-'(x) \coloneqq \sum_{\substack{y\ne x\\\sigma(x,y)=-1}} m_{\Delta}(x,y) .
\end{equation}
  To
account also for the loops, we introduce
\begin{equation} \label{eq:dplusminus}
  \begin{split} 
  d_+(x) &\coloneqq 
  \begin{cases} d_+'(x)       & \text{ if } \sigma_{\Delta}(x,x)\le 0 \\
    d_+'(x) + 2m_{\Delta}(x,x) & \text{ if } \sigma_{\Delta}(x,x)>0 ,
  \end{cases}
  \\
  d_-(x) &\coloneqq 
  \begin{cases} d_-'(x)       &\text{ if }\sigma_{\Delta}(x,x)\ge 0 \\
    d_-'(x) + 2m_{\Delta}(x,x) &\text{ if }\sigma_{\Delta}(x,x)<0 . 
  \end{cases}
  \end{split}
\end{equation}
We first state a simple property of difference graphs:
\begin{lemma}
  Every difference graph $\Delta$ satisfies 
  $d_+(x)=d_-(x)$ for all $x\in V$, and hence
  $\deg_{\Delta}(x)$ is always even.
\end{lemma}
\begin{proof}
  Since $\alpha$ is a degree-preserving bijection, we have
  $\deg_G(x)-\deg_H(\alpha(x))=0$ and thus
  \begin{align*}
  0 &= \sum_{y\in V\setminus\{x\}}
  m(\alpha(x),\alpha(y))-m(x,y) + 2 \big( m(\alpha(x),\alpha(x))-m(x,x) \big) \\ 
  &=  \sum_{y\in V\setminus\{x\}} \sigma_{\Delta}(x,y) \cdot m_{\Delta}(x,y) +
  2\sigma_{\Delta}(x,x) \cdot m_{\Delta}(x,x) \\
  &= d'_+(x)-d'_-(x) + 2\sigma_{\Delta}(x,x) \cdot m_{\Delta}(x,x) ,
  \end{align*}
  where we separate the positive and negative edges in the sum.  If
  $d'_+(x)>d'_-(x)$, then $\sigma_{\Delta}(x,x)=-1$ and further
  $d_+(x)=d_+'(x)$ and $d_-(x)=d_-'(x)+2m_{\Delta}(x,x)$, which yields
  $d_+(x)=d_-(x)$. An analogous argument applies if $d'_+(x)<d'_-(x)$.
  Finally, if $d'_+(x)=d'_-(x)$, then $m_{\Delta}(x,x)=0$. Taken together,
  we always have $d_+(x)=d_-(x)$. Moreover, $\deg_{\Delta}(x)=
  d_+(x)+d_-(x) = 2d_+(x)$, and hence $\deg_{\Delta}(x)$ is even.
\end{proof}

We define a \emph{walk} on a multigraph as a sequence of distinct edges
such that any two successive edges in the sequence share a vertex. A walk
on a multigraph with sign function $\sigma\colon E\to\{+1,-1\}$ is
\emph{alternating} if consecutive edges have opposite signs, i.e.,
$\sigma(e_k) \cdot \sigma(e_{k+1})=-1$ for all~$k$ (mod $n$).

\section{Existence of Alternating Closed Walks}

In order to investigate the difference multigraph in some more detail, it
will be useful to consider an equivalent simple graph.
\begin{definition}
  \label{def:auxgraph}
  For every signed multigraph $\Delta$ denote by $A(\Delta)$ the simple
  graph obtained by subdiving each edge of $\Delta$ by the insertion of two
  \emph{subdivision vertices}. All other vertices of $A(\Delta)$ will be a
  referred to as $\Delta$-\emph{vertices}. The sign $\sigma(e')$ of an edge
  in $A(\Delta)$ that is incident to a $\Delta$-vertex equals the sign
  $\sigma(e)$ of the edge from which $e'$ derives by subdivision. The sign
  $\sigma(e'')$ of an edge $e''$ connecting two subdivision vertices is the
  opposite of the edge from which it derives, i.e.,
  $\sigma(e'')=-\sigma(e)$.
\end{definition}
In other words, each edge $e$ in $\Delta$ is replaced by a path $P_3(e)$ of
length~$3$ in $A(\Delta)$. Moreover, the two terminal edges of $P_3(e)$ are
labeled with $\sigma(e)$, while the middle edge of $P_3(e)$ is assigned
$-\sigma(e)$, see Fig.~\ref{fig:P3} for an illustration.

\begin{figure}
  \begin{tabular}{cc}
    \begin{minipage}{.5\textwidth}
      \includegraphics[width=\linewidth]{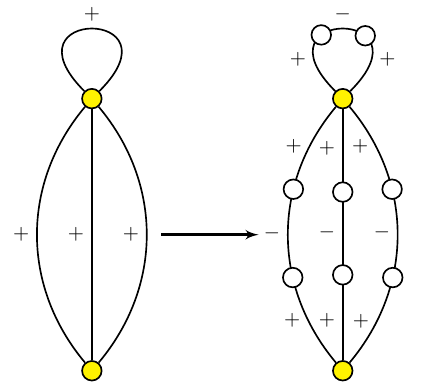}
    \end{minipage}
    &
    \begin{minipage}{.47\textwidth}
      \caption{Transformation of a signed multigraph $\Delta$ into a signed
        auxiliary graph $A(\Delta)$. Every edge of $\Delta$ is replaced by a
        a signed path $P_3$. At each $\Delta$-vertex, the number of
        positive and the number of negative edges remains unchanged.}
      \label{fig:P3}
    \end{minipage}
  \end{tabular}
\end{figure}

\begin{lemma}
  \label{lem:auxgraphwalk}
  Let $\Delta$ be a signed multigraph. If $W^*$ is a closed walk in
  $A(\Delta)$ then there is uniquely defined closed walk $W$ in $\Delta$
  such that $W^*$ is obtained from $W$ by replacing each edge $e\in W$ by
  the path $P_3(e)$. Moreover, if $W^*$ is alternating, then $W$ is also
  alternating.
\end{lemma}
\begin{proof}
  First we note that the edge set of $A(\Delta)$ can be partioned into the
  edge sets of $P_3(e)$ for $e\in E(\Delta)$. Furthermore, since the
  subdivision vertices have degree two in $A(\Delta)$, we have either
  $P_3(e)\subseteq W^*$ or $P_3(e)\cap W^*=\emptyset$ (with sequences
  treated as sets). Since $P_3(e)$ is incident with $\Delta$-vertex $x$ in
  $A(\Delta)$ if and only if $e$ is incident with $x$ in $\Delta$, every
  closed walk in $A(\Delta)$ corresponds to a closed walk in $\Delta$ by
  contracting $P_3(e)$ to the single edge~$e$.

  Consider two consecutive edges in $W^*$. By assumption, they have
  opposite sign. If they lie in the same $P_3(e)$, then one of them is the
  middle edge of $P_3(e)$. Otherwise, they lie in two different paths
  $P_3$, say $e'_1$ in $P_3(e_1)$ and $e'_2$ in $P_3(e_2)$, and the vertex
  $x$ between them is a $\Delta$-vertex. Thus, $\sigma(e'_1)=\sigma(e_1)$
  and $\sigma(e'_2)=\sigma(e_2)$, and $\sigma(e_1)=-\sigma(e_2)$ implies
  $\sigma(e'_1)=-\sigma(e'_2)$.  That is, $W$ is an alternating closed walk
  in $\Delta$.
\end{proof}
As an immediate consequence, we observe:
\begin{corollary} \label{cor:corr}
  Let $\Delta$ be a signed multigraph and $A(\Delta)$ be the corresponding
  simple graph.  The edge set of $\Delta$ is the disjoint union of the edge
  sets of alternating closed walks on $\Delta$ if and only if the edge set
  of $A(\Delta)$ is the disjoint union of the edge sets of alternating
  closed walks on $A(\Delta)$.
\end{corollary}
  
The auxiliary graph $A(\Delta)$ retains the signs of the edges in the
following sense:
\begin{lemma}
  Let $\Delta$ be a signed multigraph and $A(\Delta)$ be the corresponding
  simple graph.  Every $\Delta$-vertex $x$ in $A(\Delta)$ is incident with
  $d_+(x)$ positive edges and $d_-(x)$ negative edges. Moreover, every
  subdivision vertex in $A(\Delta)$ is incident with exactly one positive
  and one negative edge.
\end{lemma}
\begin{proof}
  By construction, if $e$ is an edge between $x$ and $y\ne x$ in $\Delta$,
  then there is exactly one edge in $P_3(e)$ incident with $x$ in
  $A(\Delta)$, and if $e$ is loop at $x$, then there are exactly two edges
  incident with $x$. In each case, the edges incident with $x$ in
  $A(\Delta)$ have the same sign as those in $\Delta$.  Since every loop in
  $\Delta$ contributes two to $d_+(x)$ resp.\ $d_-(x)$,
  cf.~Eqn.~\eqref{eq:dplusminus}, the assertion follows.
\end{proof}

\begin{corollary}
  Let $G$ and $H$ be multigraphs, $\alpha\colon V(G)\to V(H)$ be a
  degree-preserving bijection, and $\Delta$ be the difference multigraph.
  Every vertex of the corresponding simple graph $A(\Delta)$ is incident to
  as many positive as negative edges.
\end{corollary}

In the following we derive a characterization of simple graphs with signed
edges in terms of alternating closed walks that is reminiscent of Euler's
solution of the K{\"o}nigsberg Bridges Problem \cite{Euler:1741}.  We start
with a simple observation: 
\begin{lemma}
  \label{lem:existsW} 
  Let $A$ be a simple graph with non-empty edge set $E$ and edge labels
  $\sigma\colon E\to\{+1,-1\}$ such that every vertex is incident to as
  many positive as negative edges.  Then there exists an alternating closed
  walk in $G$.
\end{lemma}
\begin{proof}
  We construct an alternating walk in $A$ starting at the vertex $x_0$ with
  a positive edge. At each step, we extend the walk by traversing an edge
  that has not been used before. If a vertex $y\ne x_0$ has been entered
  along an edge $e$ with sign $\sigma(e)$, then by assumption there is a
  previously unused edge $e'$ with $\sigma(e')=-\sigma(e)$ along which the
  walk can leave $y$ again.  Obviously, the walk cannot end at a vertex
  $y\ne x_0$, and since $A$ is finite, the walk eventually encounters $x_0$
  again.  If it enters $x_0$ along a negative edge, we have obtained the
  desired alternating closed walk. If $x_0$ is entered along a positive
  edge, then the walk so far contains two more positive than negative edges
  incident with $x_0$.  Repeating the above construction (starting at $x_0$
  this time with a remaining negative edge), we conclude that the walk will
  eventually return to $x_0$ along a negative edge. Thus, there is always
  an alternating closed walk.
\end{proof}

\begin{theorem}
  \label{thm:alternatingEuler} 
  Let $A$ be a simple graph with non-empty edge set $E$ and edge labels
  $\sigma\colon E\to\{+1,-1\}$.  Then $E$ is the disjoint union of
  alternating closed walks if and only if every vertex is incident to as
  many positive as negative edges.
\end{theorem}
\begin{proof}
  First we note that the condition is necessary. Suppose there is a vertex
  $x$ with $d_+(x)\ne d_-(x)$. Since every alternating closed walk $W_i$
  necessarily traverses the same number of positive and negative edges,
  successively removing all edges in the walks $W_i$ leaves an excess of
  positive or negative edges at $x$. Thus $E \setminus \bigcup_i
  W_i\ne\emptyset$, contradicting the assumption that the $W_i$ form a
  partition of $E$.

  In order to see that the condition is sufficient, we use
  Lemma~\ref{lem:existsW} to establish the existence of an alternating
  closed walk in $A$, whose edge set we denote by $W$. Denote by
  $A'=(V,E\setminus W)$ the graph obtained by deleting the edges along the
  alternating closed walk. Since every vertex $x\in V$ is incident to as
  many positive as negative edges in both $E(A)$ and $W$, this is also true
  for $E(A')=E \setminus W$. The graph $A'$ therefore contains an
  alternating closed walk $W'$.  Repeating the argument yields a partition
  of $E$ into the edge sets of alternating closed walks.
\end{proof}

The connection to Eulerian graphs is that the vertex degree of the graph in
Thm.~\ref{thm:alternatingEuler} is even, and thus the graph is Eulerian if
and only if it connected. Assuming that $A$ is connected, one can combine
the alternating closed walks to a single alternating closed walk, see
e.g.\ \cite{Fleischner:90}, arriving at the following variant of
Hierholzer's algorithm \cite{Hierholzer:1873}: If two alternating closed
walks $W_1$ and $W_2$ share a vertex $x$, we first follow $W_1$ from an
arbitrary starting point to $x$, and then traverse $W_2$ choosing the
direction such that the combined walk is alternating. After returning to
$x$, we follow the unused part of $W_1$ from $x$ to the starting
point. This results in a single alternating closed walk $W'$ covering
exactly the edges of $W_1$ and $W_2$. If $W'$ shares a vertex $x'$ with
some alternating closed walk $W_3$, the procedure is repeated, resulting in
a single alternating closed walk $W''$ comprising the edges of $W_1$,
$W_2$, and $W_3$. One continues until all $W_i$ are absorbed into a single
alternating closed walk. Connectedness of $A$ implies that this final walk
contains all edges and thus is an \emph{alternating Euler tour}.
Otherwise, Hierholzer's algorithm produces an alternating Euler tour for
each connected component of $A$.  This argument implies:
\begin{corollary}
  Let $A$ be a simple graph with non-empty edge set $E$ and edge labels
  $\sigma \colon E\to\{+1,-1\}$. Then $A$ admits an alternating Euler tour
  if and only if $A$ is connected and every vertex is incident to as many
  positive as negative edges.
\end{corollary}

\section{Explanations by Electron Pushing Diagrams} 

Returning to atom-atom maps (AAMs) and corresponding difference multigraphs
$\Delta(G,H,\alpha)$, we can use Cor.~\ref{cor:corr} to restate
Thm.~\ref{thm:alternatingEuler}:
\begin{corollary}
  The edge set of a difference multigraph $\Delta$ can be partitioned into
  the edge sets of alternating closed walks. Moreover, if $\Delta$ is
  connected, it admits an alternating Euler tour.
\end{corollary}

Every alternating closed walk can be directly interpreted as an electron
pushing diagram (EPD). To this end, choose an arbitrary vertex $x$ and a
starting edge $e=xy$ with $\sigma(e)=-1$. The corresponding electron pair
is considered to be ``pushed'' from $xy$ to the next edge $yz$ along the
walk. Clearly, every EPD interpretation of an alternating closed walk
applied to the educt graph $G$ yields the same graph $G'$, describing a
reaction intermediate.  Thus, we can derive an \emph{EPD explanation} for
every AAM as follows: First, construct the difference multigraph
$\Delta\coloneqq\Delta(G,H,\alpha)$, then choose a sequence of alternating
closed walks whose edge sets partition $E(\Delta)$. The sequences of
negative and positive edges of each walk, starting from any of its
vertices, form an EPD. The AAM is explained by applying these EPDs in the
chosen order to the educt multigraph $G$. Of course, the order of the
closed walks can be chosen arbitrarily and yields the same product graph.

By definition, alternating closed walks have the property that every edge
is traversed only once. By construction, all edges of $\Delta$ connecting
two given vertices $x$ and $y$ have the same sign. Hence, the EPD
explanation of the AAM is monotone in the sense that the bond order, i.e.,
the multiplicity $m(x,y)$, is either always increased or always decreased
when a alternating closed walk traverses $xy$. In particular, each
alternating closed walk only traverses edges that are present in the educt
multigraph, the product multigraph, or both. The explanations produced in
this manner never insert a new bond and later delete it again. However, the
walks in this construction may be very long, and thus chemically
unrealistic.

\begin{figure} 
  \begin{center}
    \begin{tabular}{rcl}
      \begin{minipage}{0.5\textwidth}
      \includegraphics[width=\textwidth]{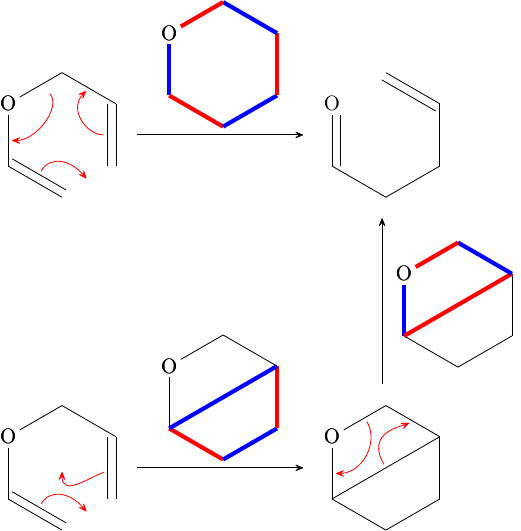}
      \end{minipage}
      &&
      \begin{minipage}{0.40\textwidth}
        \caption{Claisen rearrangement of vinyl-allyl ethers. The
          6-cyclic imaginary transition state (upper path), can be decomposed
          into two 4-cyclic imaginary transition states (lower path) to
          yield the $\gamma,\delta$-unsaturated carbonyl compound.}
        \label{fig:reallife}
      \end{minipage}
    \end{tabular}
  \end{center}
\end{figure}

We ask, therefore, if one can produce an EPD explanation with restricted
lengths of the alternating closed walks. The latter have an even length
with a minimum of $2k=4$ edges.  (The case of length~$2$ is impossible. It
corresponds to inserting and subsequently deleting the same edge, which
does not change the graph to which it is applied.)

Let $W$ be an alternating closed walk of length $|W|=2k\ge 6$ with the
first edge being negative, and let $\hat W=(x_0,x_1,\dots,x_{2k-1},x_0)$ be
the sequence of traversed vertices.  Note that edges $x_{2j}x_{2j+1}$ are
negative, and edges $x_{2j+1}x_{2j+2}$ (mod 2$k$) are positive.  The EPD
(represented by) $\hat W$ is applied to the graph $G$ which yields $H$.
Now, consider the EPDs
\begin{equation}
  \label{eq:splitW}
  \begin{split} 
    \hat W_1 &\coloneqq  (x_0,x_1,\dots,x_{2i-1},x_0) \\  
    \hat W_2 &\coloneqq (x_0,x_{2i-1},x_{2i},\dots,x_{2k-1},x_0)
  \end{split}
\end{equation}
with $2 \le i < k$.  Note that $\hat W_1$ and $\hat W_2$ have length $2i$
and $2(k-i+1)$, respectively. The EPD $\hat{W}_1$ has the fictitious last
step $x_{2i-1}x_0$, which is positive (since the first step $x_0x_1$ is
negative), even if there is no such edge in $\Delta$ or if it has negative
sign.  Analogously, the fictitious first step $x_0x_{2i-1}$ of $\hat{W}_2$
is negative (since the last step $x_{2k-1}x_0$ is positive).  Now, if
$x_0x_{2i-1}$ is not an edge of $G$, then $\hat{W}_1$ is applicable to $G$
since all changes caused by $\hat{W}_1$ (other than $x_0x_{2i-1}$) yield
changes in multiplicities that are smaller than the changes caused by
$\hat{W}$.  Moreover, $\hat{W}_1$ adds the edge $x_0x_{2i-1}$ to $G$. Hence
$\hat{W}_2$ is applicable to the resulting graph. Furthermore, since the
union of $\hat{W}_1$ and $\hat{W}_2$ equals $\hat{W}$ except for the
insertion and deletion of the fictitious edge $x_0x_{2i-1}$, the
consecutive application of $\hat{W}_1$ and $\hat{W}_2$ recovers the graph
$H$.  Similarly, if $x_0x_{2i-1}$ is an edge in $G$, then $\hat{W}_2$ is
applicable to $G$, which deletes this edge or at least reduces the bond
order. In the second step, the bond order is re-established by applying
$\hat{W}_1$. Again, the consecutive application of $\hat{W}_1$ and
$\hat{W}_2$ recovers the graph $H$.

In summary, we can split an EPD $\hat{W}$ of length $2k$ into two parts
$\hat{W}_1$ and $\hat{W}_2$ of length $l_1$ and $l_2$ with
$l_1+l_2=2(k+1)$. In this operation, an additional bond is formed and
broken again, or \textit{vice versa}. Repeating the splitting, we can break
up any EPD into a sequence of EPDs of length~$4$. Every such operation
involves the insertion and deletion of an additional bond not present in
the original EPD.  For an illustration, see Fig.~\ref{fig:reallife}.
We have shown the first main result of this contribution:
\begin{corollary}
  \label{cor:main}
  Every atom-atom map can be explained by a finite sequence of EPDs of
  length~$4$.
\end{corollary}

A chemical reaction in which the number of non-bonding electron pairs,
i.e., the loops in the multigraph representing the Lewis formula, does not
change at any atom is called \emph{homovalent} \cite{Hendrickson:97}.  The
difference graph $\Delta$ is loop-free if and only if the reaction is
homovalent. The resulting alternating closed walks therefore are also
loop-free. Cor.~\ref{cor:main} immediately implies:
\begin{corollary}
  \label{cor:homoval}
  Every atom-atom map of a homovalent reaction can be explained by a finite
  sequence of EPDs each of which is a simple four cycle.
\end{corollary}

\section{Decomposition into Elementary Reactions}

In an \emph{elementary reaction}, one or more chemical species react
directly to form products in a single step and with a single transition
state. Collision theory implies that the probability of three or more
molecules reacting simultaneously is negligible. Hence almost all reactions
will be uni-molecular (isomerizations or decompositions) or
bi-molecular. Even classical examples of termolecular reactions such as the
oxidation of nitrogen monoxide, {2 NO+O\textsubscript{2} $\to$ 2
  NO\textsubscript{2}}, are likely explained by sequences of uni- and
bi-molecular steps \cite{Olbregts:85}. Reversibility then implies that
elementary reactions produce no more than two product molecules.
\begin{definition}
  A chemical reaction is \emph{combinatorially elementary} if it is
  of the form {$A\to B$}, $A\to B+C$, $A+B\to C$, or $A+B\to C+D$, where
  $A$, $B$, $C$, and $D$ are not necessarily distinct.
\end{definition}
\begin{figure} 
  \begin{center}
    \begin{tabular}{cc}
      \begin{minipage}{0.5\textwidth}
      \includegraphics[width=\textwidth]{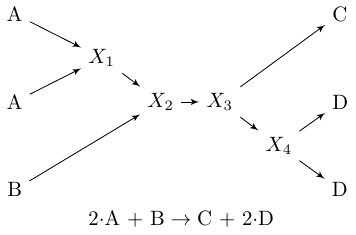}
      \end{minipage}
      &
      \begin{minipage}{0.40\textwidth}
        \caption{Representation of a complex overall reaction by elementary
          steps. Introducing intermediate species $X_i$, every
          left-hand side of an overall reaction
          (e.g.\ 2$\cdot$A~+~B$\;\rightarrow\;$C~+~2$\cdot$D) can be
          transformed into the respective right-hand side. In the proof of
          Prop.~\ref{prop:Stefan}, the educt species are first aggregated
          into a single intermediate ($X_2$ in the shown example)
          by means of 2-to-1 reactions.
          The corresponding intermediate in the product species ($X_3$) is then
          disassembled by 1-to-2 reactions.}
        \label{fig:reabybintree}
      \end{minipage}
    \end{tabular}
  \end{center}
\end{figure}
This begs the question whether an arbitrary (complex) reaction can always
be decomposed into a sequence of combinatorially elementary reactions. Next
we show that this is always possible at the abstract level:
\begin{proposition}
  \label{prop:Stefan}
  Every formal reaction $\sum a_i X_i \to \sum b_i X_i$ can be written as a
  sum of combinatorially elementary reactions. Let $a=\sum_i a_i$ and
  $b=\sum_i b_i$. Then, the number of additional intermediate species is
  bounded above by $\max(a-2,0)+\max(b-2,0)$.
\end{proposition}
\begin{proof}
  If $a,b \le 2$, the reaction itself is combinatorially elementary.  Now,
  let $a,b>2$, and let $T_a$ and $T_b$ be rooted binary trees
  with $a$ and $b$ leaves, respectively. The leaves of the two trees can be
  labeled such that the chemical species $X_i$ appears $a_i$ times as a
  leaf label of $T_a$ and $b_i$ times as a leaf label of $T_b$. Each
  non-leaf vertex of $T_a$ and $T_b$ corresponds to a chemical intermediate
  such that in $T_a$ a parent $p$ is formed from its two children $c_1$ and
  $c_2$, i.e., by the reaction $c_1+c_2\to p$, while in $T_b$ two children
  are formed from their parent, i.e., by the reaction $p \to c_1+c_2$.
  Clearly, the trees have $a-1$ and $b-1$ non-leaf vertices.  At the roots
  of the two trees, we have the intermediates $r_a$ and $r_b$,
  respectively, and the corresponding reactions $c_1 + c_2 \to r_a$ and
  $r_b \to c_3 + c_4$.  We link the two sets of reactions by adding a
  reaction between the roots of the two trees, $r_a \to r_b$, or by
  directly adding $c_1 + c_2 \to c_3 + c_4$, thereby omitting two
  intermediates.  In the latter case, we obtain a possible decomposition of
  the formal reaction into combinatorially elementary reactions with
  $n=a-2+b-2=a+b-4$ intermediates (and $n+1$ reactions among them). If
  either $a>2$ and $b\le 2$ or vice versa, then we need to consider only
  one tree. See also Fig.~\ref{fig:reabybintree}.
\end{proof}

Prop.~\ref{prop:Stefan} prompts the question whether an arbitrary AAM also
can be decomposed into combinatorially elementary reactions. As we shall
see, Cor.~\ref{cor:main} can be used to establish a stronger version of
Prop.~\ref{prop:Stefan}. To this end, we define the \emph{active parts}
$G^a$ and $H^a$ of the educt and product graphs $G$ and $H$ under an AAM
$\alpha$ as the union of the connected components of $G$ and $H$ that
contain an edge contributing to the difference graph $\Delta(G,H,\alpha)$.

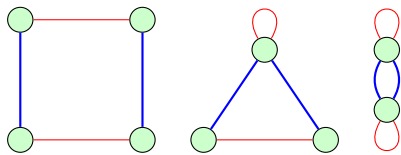
\begin{figure}
  \begin{tabular}{cc}
  \begin{minipage}{0.5\textwidth}
\begin{tikzpicture}
  [scale=0.8,auto=left]
  \tikzset{v/.style={circle,draw,fill=green!20}}
  \node[v]  (n1) at (2,2) {};
  \node[v]  (n2) at (2,4) {};
  \node[v]  (n3) at (4,4) {};
  \node[v]  (n4) at (4,2) {};
  \draw[thick,blue]  (n1) -- (n2);
  \draw[thin,red]    (n2) -- (n3);
  \draw[thick,blue]  (n3) -- (n4);
  \draw[thin,red]    (n4) -- (n1);

  \node[v]  (n1) at (5,2) {};
  \node[v]  (n2) at (6,3.5) {};
  \node[v]  (n3) at (7,2) {};
  \draw[thick,blue]  (n1) -- (n2);
  \path (n2) edge[out=60,in=120,looseness=9,thin,red] (n2); 
  \draw[thick,blue]  (n2) -- (n3);
  \draw[thin,red]    (n3) -- (n1);

  \node[v]  (n1) at (8,2.5) {};
  \node[v]  (n2) at (8,3.5) {};
  \path (n1) edge[out=60,in=300,looseness=1,thick,blue] (n2);
  \path (n1) edge[out=120,in=240,looseness=1,thick,blue] (n2);  
  \path (n2) edge[out=60,in=120,looseness=9,thin,red] (n2);
  \path (n1) edge[out=240,in=300,looseness=9,thin,red] (n1);
\end{tikzpicture}
  \end{minipage}
  &
  \begin{minipage}{0.47\textwidth}
    \caption{There are three types of alternating closed walks with four
      edges. (Positive edges in blue, negative edges in red.) The simple
      4-cycle (left) corresponds to the simplest homovalent reaction
      mechanism. The triangular walk (middle) is the simplest example of an
      ambivalent mechanism, changing the oxidation state of the atom at the
      top.}
    \label{fig:4walks}
  \end{minipage}
  \end{tabular}
\end{figure}

\begin{lemma}
  Let $\alpha \colon V(G)\to V(H)$ be an AAM explained by a single EPD of length~$4$. 
  Then the active parts $G^a$ and $H^a$ each comprise at most two
  connected components.
  \label{lem:4to2}
\end{lemma}
\begin{proof}
  An EPD of length~$4$ corresponds to an alternating closed path of
  length~$4$.  It is easy to see that there are exactly three types of
  alternating closed walks of length~$4$, see Fig.~\ref{fig:4walks}.
  \begin{itemize}
  \item[(i)] $\Delta$ is a simple cycle of size $4$ with alternating signs.
    In particular, both the positive and the negative edges form
    a matching.
  \item[(ii)] $\Delta$ is a triangle with one vertex carrying a loop. The
    edge not incident to that vertex has the same sign as the
    loop.
  \item[(iii)] $\Delta$ is a pair of vertices connected by two edges and
    with a loop at each vertex. Both, the two edges and the two loops,
    respectively, have the same sign.
  \end{itemize}
  Clearly, the negative edges must be present in the educt graph $G$, while
  positive edges imply that the incident vertices are connected in the
  product graph $G$. In case (i), the active vertices form two pairs in
  both $G^a$ and $H^a$, and hence they fall into at most two connected
  components. In case (ii), at least one pair of vertices is connected, and
  hence the three vertices are located in at most two connected
  components. In case (iii), there are only two active vertices, and hence
  the statement is trivial.
\end{proof}

Case (i) in the proof corresponds to the common pattern of a cyclic
transition state. Case (ii) is the simplest example of an ambivalent
reaction mechanism, see e.g.\ \cite{Hendrickson:97}. A hypothetical example
for case (iii) would be: 
\begin{equation*}
\schemestart
\chemfig[atom sep=2em]{\charge{90=\|,270=\|}{S}=\charge{90=\|,270=\|}{S}} +
\chemfig[atom sep=2em]{\charge{90=\|,270=\|,0=\|}{S}}
\arrow[,0.75,.]
\chemfig[atom sep=2em]{%
  \charge{90=\|,270=\|}{S}=\charge{90=\|}{S}=\charge{90=\|,270=\|}{S}}
\schemestop
\end{equation*}
We are are not aware of a clear example of ``mechanism'' (iii), which also
does not appear in Hendrickson's classification
\cite{Hendrickson:97}. However, if at least one of the two focal atoms, say
the lower atom $x$, is connected to a third atom $z$, one can explain (iii)
as the superposition of two steps of type (ii), see
Fig.~\ref{fig:S3}. Here, the bond-order of $xz$ is decreased by~$1$ in the
first step and restored in the second step, hence the two-step mechanism
can be used as an explanation only if $x$ and $z$ are connected by a bond
in the reactants.

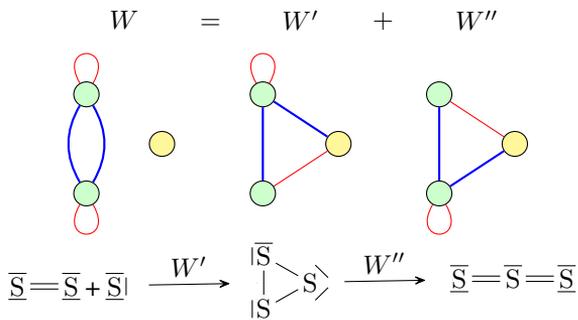
\begin{figure}
  \begin{center}
    \begin{tikzpicture}
      [scale=0.66,auto=left]
      \tikzset{v/.style={circle,draw,fill=green!20}}
\node at (0.75,6) {$W$}; \node at (4.25,6) {$W'$}; \node at (7.75,6) {$W''$};
\node at (2.25,6) {$\phantom{I}=$}; \node at (5.75,6) {$\phantom{I}+$}; 

  \node[v]  (n1) at (0,2.5) {};
  \node[v]  (n2) at (0,4.5) {};
  \node[v, fill=yellow!50]  (n3) at (1.5,3.5) {};
  \path (n1) edge[out=60,in=300,looseness=1,thick,blue] (n2);
  \path (n1) edge[out=120,in=240,looseness=1,thick,blue] (n2);  
  \path (n2) edge[out=60,in=120,looseness=9,thin,red] (n2);
  \path (n1) edge[out=240,in=300,looseness=9,thin,red] (n1);
  
  \node[v]  (n1) at (3.5,2.5) {};
  \node[v]  (n2) at (3.5,4.5) {};
  \node[v, fill=yellow!50]  (n3) at (5,3.5) {};
  \draw[thick,blue]  (n1) -- (n2);
  \path (n2) edge[out=60,in=120,looseness=9,thin,red] (n2);
  \draw[thick,blue]  (n2) -- (n3);
  \draw[thin,red]    (n1) -- (n3);
  
  \node[v]  (n1) at (7,2.5) {};
  \node[v]  (n2) at (7,4.5) {};
  \node[v, fill=yellow!50]  (n3) at (8.5,3.5) {};
  \draw[thick,blue]  (n1) -- (n2);
  \path (n1) edge[out=300,in=240,looseness=9,thin,red] (n1);
  \draw[thick,blue]  (n1) -- (n3);
  \draw[thin,red]    (n2) -- (n3);  
\end{tikzpicture}
\par\noindent

\schemestart
   \chemfig[atom sep=2em]{\charge{90=\|,270=\|}{S}=\charge{90=\|,270=\|}{S}} +
   \chemfig[atom sep=2em]{\charge{90=\|,270=\|,0=\|}{S}}
   \arrow{->[$W'$]}[1]
   \chemfig[atom sep=2em]{\charge{180=\|,270=\|}{S}*3(-\charge{45=\|,315=\|}{S}-\charge{180=\|,90=\|}{S}-)}
   \arrow{->[$W''$]}[1]
   \chemfig[atom sep=2em]{    
     \charge{90=\|,270=\|}{S}=\charge{90=\|}{S}=\charge{90=\|,270=\|}{S}}
 \schemestop
 \end{center}
  \caption{The 4-walk $W$ with 2 atoms (leftmost ITS graph) can be
    explained as a superposition of two 4-walks $W'$ and $W''$ with 3 atoms
    (middle).  This requires that at least one of the two atoms (green) is
    connected to a third atom (yellow) since in the two-step mechanism one
    bond order is first decreased and restored in the second step.  In the
    hypothetical trisulfur example, the yellow atom is one of the two
    sulfur atoms in \protect\chemfig{S_2}.  Application of the first
    alternating closed 4-walk yields the cyclic \protect\chemfig{S_3} as
    intermediate.}
\label{fig:S3}
\end{figure}

Since the connected components of $G$ and $H$ are the educt and product
molecules, respectively, Lemma~\ref{lem:4to2} can be rephrased in the
following form:
\begin{corollary}
  Let $\alpha\colon V(G)\to V(H)$ be an AAM explained by a single EPD of
  length~$4$.  Then the restriction $\alpha_a\colon V(G^a)\to V(H^a)$ is a
  combinatorially elementary reaction.
\label{cor:AAM-elem}
\end{corollary}
Corollaries~\ref{cor:AAM-elem} and \ref{cor:main} imply our main result,
which can be seen as a \emph{mechanistic} refinement of
Prop.~\ref{prop:Stefan}.
\begin{theorem}
  Let $\alpha\colon V(G)\to V(H)$ be an AAM for the reaction
  $\mathfrak{C}(G)\to\mathfrak{C}(H)$. Then there is a finite sequence of
  intermediate graphs $G=G_0,G_1,\dots,G_k=H$ and AAMs $\alpha_i\colon
  V(G_{i-1})\to V(G_i)$ such that (i) each $\alpha_i$ is explained by an
  EPD of length~$4$ and hence is a combinatorially elementary reaction, and
  (ii) $\alpha = \alpha_1\circ\alpha_2\circ\dots\circ\alpha_{k}$.
  \label{thm:main}
\end{theorem}

\section{Conclusions}

In \cite{Mueller:22a} we have shown that every reversible reaction network
has a ``chemical representation'' in terms of Lewis formulas (i.e.,
multigraphs representing molecules and AAMs for each reactions) if and only
if it is conservative.\footnote{A chemical reaction network is conservative
if its stochiometric matrix has a strictly positive left kernel
vector~\cite{Horn:72a}.}  Here we extend this representability result by
showing that that every degree-preserving bijection, and thus every AAM
$\alpha$ between Lewis formulas $G$ and $H$, can be explained by a sequence
of (edge-disjoint) alternating closed walks on the difference multigraph
$\Delta(G,H,\alpha)$. Each of these walks in turn can be obtained by a
suitably ordered sequence of electron pushing diagrams of size $4$, which
finally correspond to combinatorially elementary reactions. In particular,
the ``abstract'' Prop.~\ref{prop:Stefan} becomes an corollary of the
``mechanistic'' Thm.~\ref{thm:main}. We conclude that every conservative
reaction network admits a representation in terms of multigraphs such that
all reactions can be decomposed into elementary reactions with short
alternating closed walks, i.e., short cyclic EPDs. Most known elementary
reactions can be described by EPDs of lenght $6$, with most other cases having
length $4$ or $8$ \cite{Fujita:86b}.

The proofs given here are constructive and hence translate directly into
algorithms for constructing such a representation. First, Hierholzer's
algorithm can be used to decompose $E(\Delta)$ into a set of
(edge-disjoint) alternating closed walks $W_i$ with $O(|E(\Delta)|)$ effort
\cite{Fleischner:90}. The decomposition of the $W_i$ into EPDs of
length~$4$ can also be achieved in $O(|E(\Delta)|)$ total time. Applying
each EPD requires only constant time, and their number is again bounded by
$O(|E(\Delta)|)$. Thus the entire decomposition can be computed in linear
time.

Of course, a decomposition computed in this manner will in general not be a
chemically plausible explanation of the atom-atom map of a complex
reaction. For instance, it is possible to formally decompose the EPD of the
Diels-Alder reaction in Fig.~\ref{fig:DA} into two 4-cycles that are
chemically infeasible. The key message of this contribution, in fact, is
not to provide a method for inferring the true reaction mechanism or to
help with the inference of atom-atom maps. Instead, our results imply that
the basic \emph{combinatorial} properties of chemical reaction mechanisms
do not impose any fundamental constraints on the feasibility of
transformations as long as the preservation of atoms (and charges) is
guaranteed.

We considered here only transformations of multigraphs with loops that can
be described by relocating individual edges, i.e., electron pairs, since
these are already sufficient to handle reactions between neutral
molecules. In principle, the formalism could also be extended to
transformation that change vertex degree and thereby introduce charges at
individual atoms:
\begin{equation*}
\schemestart
\chemfig[atom sep=2em]{{{(CH_3)}_3}C-Br}
\arrow[,0.8,.]
\chemfig[atom sep=2em]{{{(CH_3)}_3}C^{\oplus}}
+
\chemfig[atom sep=2em]{Br^{\ominus}}
\schemestop
\end{equation*}
In general, the charge $\ch(x)$ at atom $x\in V$ is given as
$\ch(x)\coloneqq \val(x) - \deg(x)$, where $\val(x)$ is number of
outer-shell electrons of atom $x$. In order to account for radicals, i.e.,
unpaired electrons, multigraphs can be extended to allow semi-edges, which
in contrast to loops have only one end attached to a vertex.  Such
constructs have been studied in particular in the context of graph covers,
see e.g.\ \cite{Bok:23} and the references therein. However, a detailed
study of charged molecules and radicals is beyond the scope of this
contribution.

\subsection*{Acknowledgements}

This work was supported in part by the Novo Nordisk Foundation (grant
NNF21OC0066551 ``MATOMICS'' to CF and PFS) and the Austrian Science Fund
(FWF, grant P33218, to SM).

\bibliographystyle{adamjoucc}
\bibliography{cycletrans}
\label{sect:biblio}

\end{document}